\documentclass[11pt,a4paper,twoside]{amsart}
\usepackage{amssymb,amsmath,amsthm,graphicx,color}
\theoremstyle{plain}
\newtheorem{theorem}{Theorem}[section]
\newtheorem{proposition}[theorem]{Proposition}

\newtheorem{lemma}[theorem]{Lemma}

\theoremstyle{remark}
\newtheorem{remark}[theorem]{Remark}

\allowdisplaybreaks[4]
\numberwithin{equation}{section}

\newcommand{\R}{\mathbb{R}}

\renewcommand{\Re}{\operatorname{Re}}
\newcommand{\I}{\infty}

\newcommand{\Jbr}[1]{\langle #1 \rangle}
\newcommand{\JBR}[1]{\left\langle #1 \right\rangle}
\newcommand{\norm}[1]{\left\lVert #1\right\rVert}

\newcommand{\tnorm}[1]{\lVert #1\rVert}

\def\({\left(}
\def\){\right)}
\def\<{\left\langle}
\def\>{\right\rangle}
\def\le{\leqslant}
\def\ge{\geqslant}

\def\l{\lambda}

\newcommand{\eps}{\varepsilon}

\newcommand{\rre}{\mathbb{R}}
\newcommand{\pt}{\partial}

\begin{document}

\title[Critical  
nonlinear Klein-Gordon equation in 3D]
{Modified scattering for the Klein-Gordon\\
equation with the critical nonlinearity\\
in three dimensions}
\author[S.Masaki and J.Segata]
{Satoshi Masaki and Jun-ichi Segata}

\address{Department systems innovation\\
Graduate school of Engineering Science\\
Osaka University\\
Toyonaka Osaka, 560-8531, Japan}
\email{masaki@sigmath.es.oasaka-u.ac.jp}

\address{Mathematical Institute, Tohoku University\\
6-3, Aoba, Aramaki, Aoba-ku, Sendai 980-8578, Japan}
\email{segata@m.tohoku.ac.jp}

\subjclass[2000]{Primary 35L71; Secondary 35B40, 81Q05}

\keywords{scattering problem}

%\vskip-5mm 
%\centerline{${}^{\dagger}$ Department systems innovation, Graduate school of Engineering Science,}
%\centerline{Osaka University, Toyonaka Osaka, 560-8531, Japan}
%\centerline{Email: masaki@sigmath.es.oasaka-u.ac.jp}

%\vskip3mm 
%\centerline{${}^{\ddagger}$ Mathematical Institute, Tohoku
%University} 
%\centerline{ 6-3, Aoba, Aramaki, Aoba-ku, Sendai 980-8578,
%Japan}
%\centerline{Email: segata@m.tohoku.ac.jp}

%\vskip-15mm
\begin{abstract}
In this paper, we consider the final state problem  
for the nonlinear Klein-Gordon equation 
(NLKG) with a critical nonlinearity in three  
space dimensions:
% \begin{eqnarray}
$(\Box+1)u=\lambda|u|^{2/3}u$, $t\in\rre$, $x\in\rre^{3}$,
% \label{k1}
% \end{eqnarray}
where $\Box=\pt_{t}^{2}-\Delta$ is d'Alembertian. 
We prove that for a given asymptotic profile $u_{\mathrm{ap}}$, 
there exists a solution $u$ to (NLKG) 
which converges to $u_{\mathrm{ap}}$ as $t\to\infty$.
Here the asymptotic profile $u_{\mathrm{ap}}$ is given by the leading term of the solution 
to the linear Klein-Gordon equation with a logarithmic phase correction.
Construction of a suitable approximate 
solution is based on the combination of Fourier series expansion for 
the nonlinearity used in our previous paper \cite{MS2} and smooth modification of 
phase correction by Ginibre-Ozawa \cite{GO}.
\end{abstract}

\maketitle

\section{Introduction}
This paper is devoted to the study of the final state problem 
for the nonlinear Klein-Gordon equation with a critical nonlinearity in three space dimensions:
\begin{equation}
%\begin{eqnarray}
\left\{
\begin{array}{l}
\displaystyle{(\Box+1)u
=\lambda|u|^{2/3}u
\qquad t\in\rre,\ x\in\rre^{3},}\\
\displaystyle{
u-u_{\mathrm{ap}}\to0\qquad\text{in}\ L^{2}\ \ \text{as}\ t\to+\infty,}
\end{array}
\right. \label{K}
%\end{eqnarray}
\end{equation}
where $\Box=\pt_{t}^{2}-\Delta$ is d'Alembertian, 
$u:\rre\times\rre^{3}\to\rre$ is an unknown
function, $u_{\mathrm{ap}}:\rre\times\rre^{3}\to\rre$ is a given function, 
and $\lambda$ is a non-zero real constant. 
The aim of this paper is to find a proper choice of the function $u_{\mathrm{ap}}$ so that
the equation \eqref{K} admits a nontrivial solution.
In other words, we want to determine a ``right'' asymptotic behavior which actually takes place.
This is a continuation of our previous study of the two dimensional case in \cite{MS3}.

Let us briefly review the known results on the 
global existence and long time behavior of solution 
to the more general nonlinear Klein-Gordon equation
\begin{equation}
(\Box+1)u=\lambda|u|^{p-1}u,\qquad t\in\rre,\ x\in\rre^{d},
\label{KL}
\end{equation}
where $p>1$ and $\lambda\in\rre\backslash\{0\}$. 
%Here, we focus on the results on scattering for small data.
%For the scattering results for large data, 
%see \cite{B1,B2,GV,Na} for instance. 
%For the case $p_{0}(n)<p\le1+4/(n-1)$ with  
%$p_{0}(n)=(n+2+\sqrt{n^{2}+12n+4})/(2n)$, 
%small data scattering for (\ref{KL}) was 
%studied by many authors, 
Since the point-wise decay of solution to the linear 
Klein-Gordon equation is $O(t^{-d/2})$ as $t\to\infty$, 
the linear scattering theory indicates that 
the power $p=1+2/d$ will be a borderline between 
the short and long range scattering 
theories. This formal observation was firstly 
justified by 
Glassey \cite{G}, Matsumura \cite{Ma} and Georgiev 
and Yordanov \cite{GY} for $p\le1+2/d$. 
More precisely, they proved that solutions to (\ref{KL}) 
do not scatter to the solution to the linear Klein-Gordon 
equation if $1<p\le1+2/d$. 
Later, Georgiev and Lecente \cite{GL} 
obtained a point-wise decay estimates 
for small solutions to the (\ref{KL}) for $p>1+2/d$ with $d=1,2,3$ 
by using the vector field approach by 
Klainerman \cite{Kla}. % and the decay estimate of the solution 
%to the linear Klein-Gordon equation derived 
%by Georgiev \cite{Ge}. 
Furhtermore, Hayashi and Naumkin \cite{HN4} proved that 
small solutions to (\ref{KL}) scatter to the solution to 
the linear Klein-Gordon equation if $p>1+2/d$ and $d=1,2$.  
%by refining the vector field approach. 
Notice that it is an still open problem for the asymptotic behavior 
of small solution to (\ref{KL}) when $p$ is close to $1+2/d$ and 
$n\ge3$. See \cite{HN6,Kla,P1,P2,Shatah,St} for the small data scattering when 
$n\ge3$ and $p$ is large. 

For the critical case $p=1+2/d$ and $d=1$,  
% For the Klein-Gordon equation with the cubic 
% nonlinearity in one dimension, 
Georgiev and Yordanov \cite{GY} studied the 
point-wise decay of a solution to the initial value problem. %for 
% (\ref{KL}). %with $p=3$ and $n=1$. 
Delort \cite{D} obtained 
an asymptotic profile of a global solution to the equation.  %to (\ref{KL}) with $p=3$ and $n=1$ 
His proof is based on hyperbolic coordinates and 
the compactness of the support of the initial data was assumed. 
See also Lindblad and Soffer \cite{LS2} for the alternative proof 
of \cite{D}. 
%obtained a similar result 
%for (\ref{KL}) with $\lambda<0$ and the large initial 
%data with compact support. 
The compact support assumption in \cite{D} was later
removed by Hayashi and Naumkin in \cite{HN5} by 
using the vector field approach. 

% The equation \eqref{KL} with $p=1+2/d$ and $d\ge2$
% was out of scope in the previous works due to the lack of smoothness 
% of the nonlinear term. 

Recently, 
the authors \cite{MS3} consider (\ref{KL}) with $p=1+2/d$ and $d=2$ and
specify an asymptotic profile $u_{\mathrm{ap}}$ that allows
a unique solution $u$ which converges to $u_{\mathrm{ap}}$ as $t\to\infty$.
The asymptotic profile $u_{\mathrm{ap}}$ has the same form as in the $d=1$ case.
Namely, it is the leading term of the solution 
to the linear Klein-Gordon equation with a logarithmic phase correction.
The key ingredient is to extract a resonance term by means of Fourier series expansion of the nonlinearity.
% Let us go back to our problem (\ref{K}).
In this paper, we consider \eqref{K}, that is,
a similar final value problem for (\ref{KL}) with $p=1+2/d$ and $d=3$. 
Because the power becomes a fractional number, the argument in the two dimensional case \cite{MS3}
is not directly applicable.
To deal with the nonlinearity, we use the argument in Ginibre-Ozawa \cite{GO}.

%Especially, we consider the final state problem: For a given 
%asymptotic profile $u_{\mathrm{ap}}$, we construct a solution to (\ref{K}) 
%which converges to the given asymptotic profile as $t\to\infty$. 

\medskip 
Let us introduce the asymptotic profile $u_{\mathrm{ap}}$ which we work with.
To this end, we first recall that 
the leading term of a solution to the linear Klein-Gordon equation 
\begin{equation*}
\left\{
\begin{array}{l}
\displaystyle{(\Box+1)v
=0
\qquad t\in\rre,\ x\in\rre^{3},}\\
\displaystyle{
v(0,x)=\phi_{0}(x),\quad \pt_{t}v(0,x)=\phi_{1}(x)\qquad x\in\rre^{3}}
\end{array}
\right. 
\end{equation*}
is given by 
\[
	t^{-\frac32} {\bf 1}_{\{|x|<t\}}(t,x) \Jbr{\mu}^\frac32 \rho(\mu) \Re e^{i (\Jbr{\mu}^{-1} t +  \beta(\mu))},
\]
where $\mu=\mu(t,x):=x/\sqrt{t^{2}-|x|^{2}}$, ${\bf 1}_{\Omega}(t,x)$ 
is the characteristic function supported on
$\Omega\subset\rre^{1+3}$, 
and 
$\rho\ge 0$ and $\beta \in [0,2\pi)$ are given by the relation
\[
	\rho(\mu)e^{i\beta(\mu)} = e^{-i\frac\pi4}(\Jbr{\mu}\hat{\phi}_{0}(\mu)- i\hat{\phi}_1(\mu)),
\]
see \cite{H} for instance.

For given final state $(\phi_{0},\phi_{1})$, we define the asymptotic 
profile $u_{\mathrm{ap}}$ by
\begin{equation}\label{up}
	u_{\mathrm{ap}}(t,x):=
	t^{-\frac32} {\bf 1}_{\{|x|<t\}}(t,x) \Jbr{\mu}^\frac32 \rho(\mu) \Re e^{i (\Jbr{\mu}^{-1} t + 
	\Psi(\mu)\log t + \beta(\mu))},
\end{equation}
where the phase correction term is given by
\begin{equation}
\Psi(\mu)=-  \frac{\l \Gamma(\frac{11}6) }{\sqrt\pi 
	\Gamma(\frac{7}3)} \rho(\mu)^{\frac23}.
\label{phase}
\end{equation}
Remark that the coefficient comes from the first Fourier-cosine coefficient of a $2\pi$-periodic function $|\cos \theta|^{2/3} \cos \theta$.
The final state $(\phi_0,\phi_1)$ is taken from the function space $Y$ defined by 
\begin{align*}
Y:={}&\{(\phi_{0},\phi_{1})\in{{\mathcal S}}'(\rre^3)\times{{\mathcal S}}'(\rre^3);
\|(\phi_{0},\phi_{1})\|_{Y}<\infty\},\\
\|(\phi_{0},\phi_{1})\|_{Y}
:={}&\|\phi_{0}\|_{H_{x}^{2}}+\|x\phi_{0}\|_{H_{x}^{3}}+\|x^{2}\phi_{0}\|_{H_{x}^{4}}\\
& +\|\phi_{1}\|_{H_{x}^{1}}+\|x\phi_{1}\|_{H_{x}^{2}}+\|x^{2}\phi_{1}\|_{H_{x}^{3}}.
\end{align*} 

The main result in this paper is as follows. 

\begin{theorem} \label{thm:main} 
Let $(\phi_{0},\phi_{1})\in Y$. 
For $3/4<\gamma<5/6$,
there exists  
$\varepsilon=\varepsilon(\gamma)>0$ such that %the following properties hold: 
% If $(\phi_{0},\phi_{1})$ satisfies 
if $\norm{\Jbr{\cdot}^{3/2} \rho}_{L^\I_\mu}<\varepsilon$
then there exist $T\ge 3$ and a unique solution $u(t)$ for
the equation \eqref{K} satisfying
\begin{gather}
u \in C([T,\infty);H^{\frac12}_x), \notag \\
\sup_{t \ge T} t^{\gamma} \(\|u-u_{\mathrm{ap}}\|_{L^{\infty}((t,\infty);H_x^{\frac12})}
+ \|u-u_{\mathrm{ap}}\|_{L^{\frac{10}3}((t,\infty);L_x^{\frac{10}3})} \)
<\infty,
\label{eq:conv}
\end{gather}
where the asymptotic profile $u_{\mathrm{ap}}$ 
is defined from $(\phi_0,\phi_1)$ via \eqref{up}.
%  and $P_{1}, Q_{1}$ and $\Psi$ are 
% given by (\ref{p1}), (\ref{q1}) and (\ref{phase}), respectively.
\end{theorem}

% \begin{remark}
% Concerning the final state problem for (\ref{KL}) with the cubic nonlinearity 
% in one space dimension, 
% Hayashi and Naumkin \cite{HN2} constructed 
% a modified wave operators for (\ref{KL}) for small final data.
% Furthermore, Lindblad and Soffer \cite{LS1} showed existence of a 
% modified wave operators for (\ref{KL}) for
% large final data in the case where $\lambda<0$. 
% \end{remark}

\begin{remark}
The same result holds true for equations with 
a general critical nonlinearity $F(u):\R \to \R$ satisfying
$F(\l u)=\l^{5/3} F(u)$ for all $\lambda>0$ and $u\in\R$.
See Remark \ref{rem:general} below for the detail.
\end{remark}

\begin{remark}
Concerning the scattering results for the 
Klein-Gordon equation with the critical quasilinear 
nonlinearity, 
the readers can consult Moriyama \cite{M}, Katayama \cite{Ka}, 
Sunagawa \cite{S} for one dimensional 
cubic case and Ozawa, Tsutaya and Tsutsumi \cite{OTT}, 
Delort, Fang and Xue \cite{DFX}, Kawahara and Sunagawa \cite{KS}, 
Katayama, Ozawa and Sunagawa \cite{KOS} for the two dimensional 
quadratic case. 
\end{remark}

The rest of the paper is organized as follows. 
In Section 2, we exhibit the outline of 
the proof of Theorem \ref{thm:main}.
We construct a solution with the desired property by applying 
the contraction principle to
 the integral equation of Yang-Feldman type associated with 
(\ref{K}) around a suitable 
approximate solution. %which well approximates a solution the equation (\ref{K}) for large time.
The crucial points of the proof are summarized as Propositions \ref{prop:FVP}
and \ref{prop:mainest0}.
Then, we prove Proposition \ref{prop:FVP} in Section \ref{sec:FVP}
and Proposition \ref{prop:mainest0} in Section \ref{sec:prf}

% In Section 3, we solve an abstract final 
% value problem around an approximate solution.
% Then, in Section 4, we show %our asymptotics $u_{\mathrm{ap}}$
% the approximate solution given in Section 2
% satisfies the assumptions of 
% the final value problem in Section 3,
% and we completes the proof of Theorem~\ref{thm:main}.

%%%%%%%%%%%%%%%%%%%%%%%%%%%%%%%%%%%%%%%%%%%%%%%%%%%%%%%%%%%%%%%%%%%%%

\section{Outline of the proof of Theorem \ref{thm:main}
} %\label{sec:FVP}

In this section, we give an outline of the proof of Theorem~\ref{thm:main}.

\subsection{On the solvability of the final state problem}
We first remark that solvability of \eqref{K} is reduced to the
appropriateness of the choice of the asymptotic behavior $u_{\mathrm{ap}}$.

Let $A(t,x)$ be a given asymptotic profile of a solution to (\ref{K}).
We show that if $A(t,x)$ is well-chosen then we obtain a solution which asymptotically behaves like $A(t,x)$.
Let $N(u)=\lambda|u|^{2/3}u$. 
% We introduce the \emph{error function} $\mathcal{E}(t,x)$ by
% \begin{eqnarray}
% \mathcal{E}:=(\Box+1)A -N(A).
% \label{eq:defR}
% \end{eqnarray}
%The main result of this section is the following.
\begin{proposition} \label{prop:FVP}
For any $\gamma>3/4$, there exists $\eta=\eta(\gamma)>0$ such that if 
$A(t,x)$ satisfies
\begin{gather}
\sup_{t\ge T_0} t^{\frac32} \|A(t)\|_{L^{\infty}_x} \le \eta ,
\label{eq:A1}\\
\begin{split}
\sup_{t\ge T_0} t^{1+\gamma} \|(\Box+1)A(t) -N(A(t))\|_{L^2_x}< \I
\label{eq:A2}
\end{split}
\end{gather}
for some $T_0\ge 3$,
% where $F(t,x)$ is given by (\ref{eq:defR}),
then there exist $T\ge T_0$ and 
 a unique solution $u \in C([T,\infty);L^2_x)$
for the equation (\ref{K}) satisfying
\begin{equation}
\sup_{t \ge T} t^{\gamma} \left(\|u-A\|_{L^{\infty}((t,\infty);H^{\frac12}_x)}
+\|u-A\|_{L^{\frac{10}{3}}((t,\infty);L^{\frac{10}{3}}_x)}\right) <\infty
\label{pr}
\end{equation}
for the same $\gamma$.
\end{proposition}

The proposition will be proven in Section 3.

\subsection{Choice of an appropriate asymptotic profile}

An easy choice is $A=u_{\mathrm{ap}}$.
However, it does not work well.
Hence, we choose a suitable $A$ that satisfies the assumptions \eqref{eq:A1} and \eqref{eq:A2} and
\begin{equation}\label{eq:A3}
\sup_{t \ge T} t^{\gamma} \left(\|u_{\mathrm{ap}}-A\|_{L^{\infty}((t,\infty);H^{\frac12}_x)}
+\|u_{\mathrm{ap}}-A\|_{L^{\frac{10}{3}}((t,\infty);L^{\frac{10}{3}}_x)}\right) <\infty	.
\end{equation}
Then, the solution obtained by means of
Proposition \ref{prop:FVP} from $A$ possesses the desired asymptotics.

The obstacle in three-dimensional case lies in the fact that the phase correction term $\Psi$, given in \eqref{phase},
has the fractional power term $\rho^{2/3}$. 
The power comes from the nonlinearity.
Notice that because of the fractional power, we may not estimate
$(\Box+1)u_{\mathrm{ap}}$, in general. 
To overcome the difficulty, we exploit the argument in 
Ginibre-Ozawa \cite{GO}.
We introduce a modified phase corrector
\[
	\widetilde{\Psi}(s,\mu):=
	-  \frac{\l \Gamma(\frac{11}6) }{\sqrt\pi 
	\Gamma(\frac{7}3)} \widetilde\rho(s,\mu)^\frac23, 
	\qquad \widetilde\rho(s,\mu)= \sqrt{\rho(\mu)^2+s^{-1}\Jbr{\mu}^{-3}}.
\]
and an auxiliary approximate solution
\begin{equation}\label{eq:tuap}
	\widetilde{u}_{\mathrm{ap}}(t,x)=
	t^{-\frac32} {\bf 1}_{\{|x|<t\}}(t,x) \Jbr{\mu}^\frac32 \rho(\mu) \Re e^{i (\alpha(t,\mu)+ \beta(\mu))},
\end{equation}
where $\mu=\mu(t,x):=x/\sqrt{t^2-|x|^2}$ and
$$
\alpha(t,\mu):=\Jbr{\mu}^{-1} t + \widetilde{\Psi}(t,\mu)\log t.
$$
We will see the error from the modification is acceptable.

Starting from the modified asymptotic profile $\widetilde{u}_{\mathrm{ap}}$,
we construct the profile $A$ as in the two dimensional case \cite{MS3}.
This is the idea of the proof.
For readers' convenience, we recall the construction of $A$.
Since $N(\widetilde{u}_{\mathrm{ap}}(t))$ is $O(t^{-1})$ in $L^2_x$ as $t\to\I$,
we have to, at least, find the main parts of it and cancel them out, otherwise \eqref{eq:A2} fails.
In \cite{MS3}, a Fourier series expansion is introduced for this purpose.
Here, we split %$N(\tilde{u}_{ap})$を
\begin{equation}\label{y01}
\begin{aligned}
N(\widetilde{u}_{\mathrm{ap}})
={}& \l t^{-\frac52}{\bf 1}_{\{|x|<t\}} \Jbr{\mu}^{\frac52} \rho(\mu)^{\frac53}
|\cos(\alpha+\beta)|^{\frac23} \cos(\alpha+\beta)\\
={}& \l c_1 t^{-\frac52}{\bf 1}_{\{|x|<t\}}\Jbr{\mu}^{\frac52} \rho(\mu)^{\frac53} \Re e^{i(\alpha+\beta)}\\
&{} + \l t^{-\frac52}{\bf 1}_{\{|x|<t\}}\Jbr{\mu}^{\frac52} \rho(\mu)^{\frac53} \sum_{n=3}^\I c_n \Re e^{in(\alpha+\beta)}\\
=:&N_{\mathrm{r}}+N_{\mathrm{nr}},
\end{aligned}
\end{equation}
where $c_n$ is the Fourier-cosine coefficient of $|\cos \theta|^\frac23 \cos \theta$.
We employ the following estimate on the coefficient.
\begin{lemma}[\cite{MMU}]\label{lem:MMU}
Let 
$c_n := \frac1{\pi} \int_{-\pi}^\pi |\cos \theta |^{\frac23}\cos \theta \cos n\theta d\theta$
 for $n\ge 0$.
Then, $c_n=0$ for even $n$ and 
\[
	c_n = \frac{2(-1)^{\frac{n-1}2} \Gamma(\frac{11}6) \Gamma(\frac{3n-5}6)}{\sqrt\pi 
	\Gamma(- \frac{1}3) \Gamma(\frac{3n+11}6)}
\]
for odd $n$.
In particular, $c_n = O(n^{-8/3})$ as $n\to\I$.
\end{lemma}
Thanks to the choice of the phase function $\widetilde\Psi$, 
the resonance part $N_{\mathrm{r}}$ is close to $(\Box+1)\widetilde{u}_{\mathrm{ap}}$ (See Proposition \ref{prop:mainest1}).
To cancel out $N_{\mathrm{nr}}$, we introduce
\begin{equation}\label{eq:tvap}
\widetilde{v}_{\mathrm{ap}}(t,x) := \sum_{n=2}^\I v_n(t,\mu(t,x)),
\end{equation}
with
\begin{equation}\label{eq:vn}
	v_{n}(s,\mu) := -\frac{\l c_n}{n^2-1} t^{-\frac52} \Jbr{\mu}^{\frac52} \widetilde\rho(s,\mu)^{\frac23}\rho(\mu)
	\Re (e^{in(\alpha(s,\mu)+\beta(\mu))}).
\end{equation}
It will turn out that the non-resonance part $N_{\mathrm{nr}}$ is successfully canceled out by $\widetilde{v}_{\mathrm{ap}}$ (See Proposition
\ref{prop:mainest2}).
\begin{remark}
This kind of approximation was introduced in H\"{o}rmander \cite{H} 
for the Klein-Gordon equation with \emph{polynomial} nonlinearity in $(u,\overline{u})$. 
See also \cite{MTT,ST} for the nonlinear Schr\"{o}dinger equation 
with polynomial nonlinearity in $(u,\overline{u})$. 
\end{remark}

Based on the above observation, we will show the following proposition.

\begin{proposition}\label{prop:mainest0}
Let $A=\widetilde{u}_{\mathrm{ap}}+ \widetilde{v}_{\mathrm{ap}}$, where 
$\widetilde{u}_{\mathrm{ap}}$ and $\widetilde{v}_{\mathrm{ap}}$ are given in \eqref{eq:tuap} and \eqref{eq:tvap},
respectively.
Then, \eqref{eq:A2} and \eqref{eq:A3} holds for any $\gamma<5/6$ and $T_0\ge3$.
Furthermore, for any $\eta>0$ and $\gamma<5/6$, there exists $\eps$ such that
if $\tnorm{\Jbr{\cdot}^{3/2}\rho}_{L^\I_\mu} \le \eps$ then $A$ satisfies \eqref{eq:A1} for some $T_0\ge 3$.
\end{proposition}
Together with Proposition \ref{prop:FVP}, this proposition implies 
Theorem \ref{thm:main}. 
Section \ref{sec:prf} is devoted to the proof of the above proposition.

\begin{remark}\label{rem:general}
Let us consider a generalization of Theorem \ref{thm:main} to any real-valued nonlinearity satisfying
$F(\l u)=\l^{5/3} F(u) $ for any $\l>0$ and $u\in\R$.
Notice that this class of nonlinearity
is written as $F(u)= \l_1 |u|^\frac23 u + \l_2 |u|^\frac53$.
Theorem \ref{thm:main} corresponds to the case where $\l_2=0$.
By means of the following lemma, we see that the nonlinearity $\l_2 |u|^{5/3}$ does not contain
a resonant part, and so that
we can treat the above general nonlinearity by the same argument.
\end{remark}
\begin{lemma}
Let $\widetilde{c}_n := \frac1\pi \int_{-\pi}^\pi |\cos \theta|^{\frac53} \cos n\theta d\theta$ for $n\ge0$.
Then, $\widetilde{c}_n=0$ for odd $n$ and
\[
	\widetilde{c}_n = 
	\frac{2 (-1)^{\frac{n}2}  \Gamma(\frac43)\Gamma(\frac{3n-5}6)}{\sqrt\pi \Gamma(-\frac56)\Gamma(\frac{3n+11}6)}\qedhere
\]
for even $n$. In particular, $\widetilde{c}_n = O(n^{-8/3})$ as $n\to\I$.
\end{lemma}

\begin{proof}
The proof is similar to Lemma \ref{lem:MMU}. See, \cite{MMU}.
\end{proof}

\section{Proof of Proposition \ref{prop:FVP}} \label{sec:FVP}

In this section, we prove Proposition \ref{prop:FVP}.
% we solve a Cauchy problem at infinite initial time for the equation (\ref{K})
% in an abstract framework.
The proof is essentially the same as in the two dimensional case \cite{MS3}.
The following 
inhomogeneous Strichartz estimates associated with the 
Klein-Gordon equation is crucial for the proof.
 Let 
\begin{equation}
{{\mathcal G}}[g](t)
:=\int_t^{\infty} \sin((t-\tau)\sqrt{1-\Delta})(1-\Delta)^{-1/2}g(\tau)d\tau.
\label{g}
\end{equation}

\begin{lemma}\label{lem:St} Let $2\le q\le 6$ and $2/p+3/q=3/2$. Then we have
\begin{align*}
\|{{\mathcal G}}[g]
\|_{L_{t}^{p}([T,\infty),L_{x}^{q})}
\le{}&C\|(1-\Delta)^{\frac12(\frac32-\frac5q)}g\|_{L_{t}^{p'}([T,\infty),L_{x}^{q'})},\\
\|{{\mathcal G}}[g]
\|_{L_{t}^{\infty}([T,\infty),L_{x}^{2})}
\le{}&C\|(1-\Delta)^{\frac12(\frac14-\frac5{2q})}g\|_{L_{t}^{p'}([T,\infty),L_{x}^{q'})},\\
\|{{\mathcal G}}[g]\|_{L_{t}^{p}([T,\infty),L_{x}^{q})}
\le{}&C\|(1-\Delta)^{\frac12(\frac14-\frac5{2q})}g\|_{L_{t}^{1}([T,\infty),L_{x}^{2})}.
\end{align*}
\end{lemma}

\begin{proof}
The above inequalities follow from combination of the $L^{p}$-$L^{q}$ estimate 
for the solution to the Klein-Gordon equation by \cite{MSW} with the duality 
argument by \cite{Y} for the non-endpoint case $q\neq6$ and the 
argument by \cite{KT} for the endpoint case $q=6$. 
Since the proof is now standard, we omit the detail.
\end{proof}

\begin{proof}[Proof of Proposition \ref{prop:FVP}]
We introduce 
\begin{gather*}
X_T =\{ w \in C([T,\infty);L^2_x); \ \|w\|_{X_T} < \infty \}
\end{gather*}
for $T\ge 3$, where
\begin{equation*}
\|w\|_{X_T} = \sup_{t \ge T} t^{\gamma} 
\left(\|w\|_{L^{\infty}((t,\infty);H^{\frac12}_x)}
+\|w\|_{L^{\frac{10}{3}}((t,\infty);L^{\frac{10}{3}}_x)}\right).
\end{equation*}
For $R >0$ and $T>0$, we define
\begin{gather*}
\widetilde{X}_T(R)= \{ w \in C([T,\infty);L^2_x);
\ \|w\|_{X_T} \le R \}.
\end{gather*}
The function space $X_T$ is a Banach space with the norm $\|\cdot\|_{X_T}$
and $\widetilde{X}_T(\rho)$ is a complete metric space with
the $\|\cdot\|_{X_T}$-metric.

We put $v=u-A$.
Then the equation (\ref{K}) is equivalent to
\begin{equation}
(\Box+1)v = N(v+A) -N(A) - F,
\label{K'}
\end{equation}
where
\[
	F := (\Box+1)A - N(A)
\]
The associate integral equation to the equation (\ref{K'}) is
\begin{equation}
v=
{{\mathcal G}}[\{ N(v+A) -N(A)\} -F],
\label{eq:INT}
\end{equation}
where ${{\mathcal G}}$ is given by (\ref{g}). It suffices to show the existence of a unique solution $v$
to the equation \eqref{eq:INT} in $X_T$ for suitable $\eta>0$ and $T\ge T_0$.
We prove this assertion by the contraction argument.
Define the nonlinear operator $\Phi$ by
\begin{equation*}
\Phi v:=
{{\mathcal G}}[\{ N(v+A) -N(A)\} -F]
\end{equation*}
for $v \in \widetilde{X}_T(R)$. We show that 
$\Phi$ is a contraction map on $\widetilde{X}_T(\rho)$
if $R>0$, $T\ge T_0$, and $\eta>0$ are suitably chosen.
% Let $\rho >0$ be arbitrary, and $T,\eta>0$ which will be determined below.
Let $v \in \widetilde{X}_T(R)$ and $t \ge T$.
By the assumptions and Lemma \ref{lem:St}, we see
\begin{equation*}
\begin{split}
\|(\Phi &v)(t)\|_{L^{\infty}((t,\infty);H^{\frac12}_x)}
+\|\Phi v\|_{L^{\frac{10}{3}}((t,\infty);L^{\frac{10}{3}}_x)} \\
\le & C(\||v|^{\frac23}v\|_{L^{\frac{10}{7}}((t,\infty);L^{\frac{10}{7}}_x)}\\
&
+\|(1-\Delta)^{-\frac14}\{|v+A|^{\frac23}(v+A)-|v|^{\frac23}v-|A|^{\frac23}A\}\|_{L^1((t,\infty);L^2_x)} \\
& +\|(1-\Delta)^{-\frac14}F\|_{L^1((t,\infty);L^2_x)}) \\
\le & C \biggl\{ \|v\|_{L^{\frac{10}{3}} ((t,\infty);L^{\frac{10}{3}}_x)}^{\frac23}
\left( \int_t^{\infty} \|v(\tau)\|_{L^2_x}^2 \,d\tau \right)^{\frac12} 
+\int_t^{\infty} \|A(\tau)\|_{L^{\infty}_x}^{\frac23}
\|v(\tau)\|_{L^2_x}\,d\tau \\
& +\int_t^{\infty}\|F(\tau)\|_{L^2_x}\,d\tau \biggr\} \\
\le & C \biggl\{ R^{\frac23} t^{-\frac23\gamma}
\left( \int_t^{\infty} R^2\tau^{-2\gamma} \,d\tau \right)^{\frac12}
+\int_t^{\infty} \eta^{\frac23} 
R \tau^{-1-\gamma} \,d\tau +\int_t^{\infty} M \tau^{-1-\gamma} \,d\tau \biggr\} \\
\le & C t^{-\gamma} (R^{\frac53} t^{-\frac{2}{3}\gamma+\frac12} +R\eta^{\frac23}
+M),
\end{split}
\end{equation*}
where $M$ is an upper bound on the right hand side of \eqref{eq:A2}.
Therefore we obtain
\begin{equation}
\|\Phi v\|_{X_T} \le C_{1} (R^{\frac53} T^{-\frac{2}{3}\gamma+\frac12} +R\eta^{\frac23}
+M).
\label{eq:into}
\end{equation}
In the same way as above, for $v_1,v_2 \in \widetilde{X}_T(R)$,
we can show
\begin{equation}
\begin{split}
\|\Phi v_1 & -\Phi v_2\|_{X_T} \\
\le & C_{2}((\|v_1\|_{X_T}^{\frac23} +\|v_2\|_{X_T}^{\frac23}) 
T^{-\frac{2}{3}\gamma+\frac12}
+\eta^{\frac23})\|v_1-v_2\|_{X_T} \\
\le & C_{2}(R^{\frac23} T^{-\frac{2}{3}\gamma+\frac12} +\eta^{\frac23}) \|v_1-v_2\|_{X_T}.
\label{eq:cont}
\end{split}
\end{equation}
We first fix $R$ so that $C_1 M \le R/2$. 
Then, using the fact that $\gamma>3/4$,
we are able to choose
a sufficiently large $T>0$ and a sufficiently small $\eta >0$ 
such that
\begin{gather*}
C_{1}(R^{\frac53} T^{-\frac{2}{3}\gamma+\frac12} +R\eta^{\frac23} +\eta)
\le R, \\
C_{2}(R^{\frac23} T^{-\frac{2}{3}\gamma+\frac12} +\eta^{\frac23}) \le \frac{1}{2},
\end{gather*}
For such $R,T,\eta$, there exists a unique solution to
the integral equation \eqref{eq:INT} in $\widetilde{X}_T(\rho)$.
The uniqueness of solutions to the equation
\eqref{eq:INT} in $X_T$ follows from
the first inequality of the estimate \eqref{eq:cont} for
solutions $v_1 \in X_T$ and $v_2 \in X_T$.
Hence 
the equation \eqref{eq:INT} has a unique solution in $X_T$.
This completes the proof of Proposition \ref{prop:FVP}.
\end{proof}

%%%%%%%%%%%%%%%%%%%%%%%%%%%%%%%%%%%%%%%%%%%%%%%%%%%%%%%%%%%%%%%

\section{Proof of Proposition \ref{prop:mainest0}} \label{sec:prf}

In this section, we prove Proposition \ref{prop:mainest0}.
Since \eqref{eq:A1} is trivial, we prove \eqref{eq:A3} and \eqref{eq:A2} in Sections \ref{subsec:A3}
and \ref{subsec:A2}, respectively, after preparing preliminary estimates in Section \ref{subsec:pre}.
Hereafter we always restrict our attention to the region $|x|<t$ and $t\ge 3$.

We introduce new variables $(s,\mu)$ by $s=t$ and $\mu=x/\sqrt{t^2-|x|^2}$.
Then, we have 
\begin{align}
\label{ptt}
\pt_{t}
&= \pt_{s} -s^{-1}\langle \mu \rangle^{2}\mu \cdot \nabla_\mu ,\\
\label{ptx}
\pt_{x_{i}}
&= s^{-1}\langle \mu \rangle \pt_{\mu_{i}}
+s^{-1}\langle \mu \rangle\mu_{i}\mu \cdot \nabla_\mu
\end{align}
and
\begin{align}
\Box ={}& \pt_s^2 - 2s^{-1} \langle \mu \rangle^{2}\mu \cdot \nabla_\mu \pt_s
- s^{-2} \Jbr{\mu}^2 \Delta_\mu 	
\label{ab2}\\
	& -3 s^{-2} \langle \mu \rangle^{2} \mu \cdot \nabla_\mu
	-s^{-2}\Jbr{\mu}^2 \sum_{1\le i,j \le 3 }\mu_i \mu_j \pt_i \pt_j
	\nonumber.
\end{align}
Also remark that
\begin{equation}\label{eq:change}
	\norm{f(t,\mu(t,x))}_{L^p_x(|x|<t)} = s^{\frac3p}\norm{\Jbr{\mu}^{-\frac5p}
	f(s,\mu)}_{L^p_\mu(\R^3)}
\end{equation}
for any $p \in (0,\I]$. 

\subsection{Preliminaries}\label{subsec:pre}

We collect preliminary estimates.
\begin{lemma}\label{lem:mainest0}
For $n,m\in \R$,
\begin{multline*}
	(\Box_{t,x}+1) (s^{-m}e^{i n \Jbr{\mu}^{-1} s})
	= -(n^2-1) s^{-m}e^{i n \Jbr{\mu}^{-1} s}
	\\ -in (2m -d)s^{-m-1} \Jbr{\mu} e^{i n \Jbr{\mu}^{-1} s}
	+m(m+1) s^{-m-2}e^{i n \Jbr{\mu}^{-1} s} ,
\end{multline*}
where $d=3$ is the spatial dimension.
\end{lemma}
\begin{proof} It follows by direct calculation.
\end{proof}

Recall that $\widetilde\rho(s,\mu)=
\sqrt{\rho(\mu)^2+s^{-1}\langle{\mu}\rangle^{-3}}$.
An elementary inequality
\begin{equation}\label{eq:trho_est}
	\max(\rho (\mu), s^{-\frac12} \Jbr{\mu}^{-\frac32}) \le \widetilde\rho(s,\mu) \le
	 \rho(\mu) + s^{-\frac12} \Jbr{\mu}^{-\frac32}
\end{equation}
will be useful.
The following will be used to estimate the error comes from the phase modification.
\begin{lemma}\label{lem:mainest1_pf1} 
For $s\ge3$, we have the following inequality:
\[
	\rho(\mu) (\widetilde\rho(s,\mu)^\frac23 - \rho(\mu)^\frac23) \lesssim
	s^{-\frac56} \Jbr{\mu}^{-\frac52},
\]
where the implicit constant is independent of $\rho$.
\end{lemma}
\begin{proof}
By a direct calculation, we obtain 
\begin{align*}
	&\rho(\mu)(\widetilde\rho(s,\mu)^\frac23 - \rho(\mu)^\frac23)\\
	&= \frac13 \rho(\mu) \int_0^1 (\rho(\mu)^2 + \theta s^{-1}\Jbr{\mu}^{-3})^{-\frac23} s^{-1}\Jbr{\mu}^{-3} d\theta\\
	&\le \frac13 \rho(\mu) \int_0^1 (\rho(\mu)^2)^{-\frac12} (\theta s^{-1}\Jbr{\mu}^{-3})^{-\frac16} s^{-1}\Jbr{\mu}^{-3} d\theta
	= C s^{-\frac56}\Jbr{\mu}^{-\frac52}
\end{align*}
for every $\mu \in \R^3$. Thus we obtain the desired inequality. 
\end{proof}

Now, we turn to the estimate of the derivatives of the modified phase part
$\widetilde\Psi(s,\rho)=-(\l c_1/2) \widetilde\rho(s,\mu)^{2/3}$.
\begin{lemma}\label{lem:mainest1_pf2}
For $s\ge3$, we have the following inequalities:
\begin{equation}\label{eq:estPsis0}
	|\pt_s \widetilde\Psi(s,\mu)| \lesssim s^{-\frac43}\Jbr{\mu}^{-1},
\end{equation}
\begin{equation}\label{eq:estPsis}
	| \rho(\mu)\pt_s \widetilde\Psi(s,\mu)| \lesssim s^{-1-\frac{5}6}\Jbr{\mu}^{-\frac52},
\end{equation}
\begin{equation}\label{eq:estPsis2}
	| \rho(\mu)(\pt_s \widetilde\Psi(s,\mu))^2| \lesssim s^{- 2 - \frac76 }\Jbr{\mu}^{-\frac72},
\end{equation}
\begin{equation}\label{eq:estPsiPsis}
	| \rho(\mu)^\frac23\widetilde\Psi(s,\mu)\pt_s \widetilde\Psi(s,\mu)| \lesssim s^{- 2 }\Jbr{\mu}^{-3},
\end{equation}
\begin{equation}\label{eq:estPsiss}
	| \rho(\mu)\pt_s^2 \widetilde\Psi(s,\mu)| \lesssim  s^{-2-\frac56}\Jbr{\mu}^{-\frac52},
\end{equation}
where the implicit constants are independent of $\rho$.
\end{lemma}

\begin{proof}
The first four inequalities (\ref{eq:estPsis0})-(\ref{eq:estPsiPsis}) 
follow from
\[
	\pt_s \widetilde\Psi(s,\mu) %= -\frac{\l c_1}{2\Jbr{\mu}}\frac13 (A(\mu)^2 + cs^{-\nu})^{-\frac23} (-c\nu s^{-\nu-1})
	=Cs^{-2} \Jbr{\mu}^{-3}\widetilde\rho(s,\mu)^{-\frac43},
\]
and \eqref{eq:trho_est}.
Similarly,
\[
	\pt_s^2 \widetilde\Psi(s,\mu) = C_1 s^{-3}\Jbr{\mu}^{-3}\widetilde\rho(s,\mu)^{-\frac43}
	+C_2 s^{-4}\Jbr{\mu}^{-6} \widetilde\rho(s,\mu)^{-\frac{10}3},
\]
yields the last inequality (\ref{eq:estPsiss}).
\end{proof}

\begin{lemma}\label{lem:mainest1_pf3}
For $s\ge3$, we have the following inequalities:
\begin{equation}\label{eq:mupsi}
	|\pt_{\mu_j} \widetilde\Psi(s,\mu)| \lesssim 
	s^{\frac16}\Jbr{\mu}^\frac12 |\pt_{\mu_j} \rho(\mu)| + s^{-\frac13}\Jbr{\mu}^{-2},
\end{equation}
\begin{equation}\label{eq:Amupsimupsi}
	|\rho(\mu)^\frac23 \pt_{\mu_j} \widetilde\Psi(s,\mu) \pt_{\mu_k} \widetilde\Psi(s,\mu)| \lesssim 
	|\nabla_\mu \rho(\mu)|^2+s^{-1}\Jbr{\mu}^{-5},
\end{equation}
\begin{equation}\label{eq:Amumupsi}
	|\rho(\mu)\pt_{\mu_j}\pt_{\mu_k} \widetilde\Psi(s,\mu)| \lesssim 
	\rho(\mu)^{\frac23} |\nabla_\mu^2 \rho(\mu)|
	+s^{\frac16} \Jbr{\mu}^\frac12|\nabla_\mu \rho(\mu)|^2+s^{-\frac13}\Jbr{\mu}^{-3}\rho(\mu),
\end{equation}
\begin{equation}\label{eq:Amuspsi}
	|\rho(\mu)\pt_{\mu_j}\pt_{s} \widetilde\Psi(s,\mu)| \lesssim s^{-\frac43}\Jbr{\mu}^{-1} |\pt_{\mu_j} \rho(\mu)|
	+s^{-\frac43}\Jbr{\mu}^{-2}\rho(\mu),
\end{equation}
where the implicit constants are independent of $\rho$.
\end{lemma}

\begin{proof}
The first two inequalities (\ref{eq:mupsi}) and 
(\ref{eq:Amupsimupsi}) follow from
\[
	\pt_{\mu_j} \widetilde\Psi %= -\frac{\l c_1}{2\Jbr{\mu}} \frac23 (A(\mu)^2 + cs^{-\nu})^{-\frac23} A(\mu)\pt_{\mu_j}A(\mu)
	=  C_1 \widetilde\rho(s,\mu)^{-\frac43} (\rho(\mu) \pt_{\mu_j} \rho(\mu) -3 s^{-1}\Jbr{\mu}^{-5}\mu_j)
\]
and \eqref{eq:trho_est}.
To obtain the inequality (\ref{eq:Amumupsi}), we use
\begin{align*}
	\pt_{\mu_j}\pt_{\mu_k} \widetilde\Psi
	={}& C_2 \widetilde\rho(s,\mu)^{-\frac{10}3} (\rho(\mu) \pt_{\mu_j} \rho(\mu) -3 s^{-1}\Jbr{\mu}^{-5}\mu_j)\\
&\qquad \qquad \qquad \times (\rho(\mu) \pt_{\mu_k} \rho(\mu) -3 s^{-1}\Jbr{\mu}^{-5}\mu_k)\\
	&{}+C_1 \widetilde\rho (s,\mu)^{-\frac{4}3} (\pt_{\mu_j} \rho(\mu)\pt_{\mu_k} \rho(\mu)+\rho(\mu)\pt_{\mu_j}\pt_{\mu_k} \rho(\mu))\\
	&{}+C_1 \widetilde\rho (s,\mu)^{-\frac{4}3}(15s^{-1}\Jbr{\mu}^{-7}\mu_j\mu_k -3 s^{-1} \Jbr{\mu}^{-5}\delta_{jk}) ,
\end{align*}
where $\delta_{jk}$ is the Kronecker delta.
The inequality (\ref{eq:Amuspsi}) is a consequence of
\begin{multline*}
	\pt_{\mu_j}\pt_{s} \widetilde\Psi
	=C_3 s^{-2}\Jbr{\mu}^{-3}\widetilde\rho (s,\mu)^{-\frac{10}3}(\rho(\mu) \pt_{\mu_j} \rho(\mu)-3 s^{-1}\Jbr{\mu}^{-5}\mu_j)\\
	+3 C_1 s^{-2}\widetilde\rho (s,\mu)^{-\frac{4}3}\Jbr{\mu}^{-5}\mu_j.
\end{multline*}
This completes the proof of Lemma \ref{lem:mainest1_pf3}.
\end{proof}

%%%%%%%%%%%%%%%%%%%%%%%%%%%%%%%%%%%%%%%%%%%%%%%%%%%%%%%%%%%%%%%
\subsection{Proof of \eqref{eq:A3}}\label{subsec:A3}

We now show that $A$ satisfies \eqref{eq:A3} for $\gamma<5/6$.
\begin{proof}
Since
\begin{equation*}
	|e^{-\frac{ic_1\l}{2} \widetilde\rho(s,\mu)^\frac23\log t}
	- e^{-\frac{ic_1\l}{2} \rho(\mu)^\frac23 \log t}|
	\lesssim C(\widetilde\rho(s,\mu)^\frac23 - \rho(\mu)^\frac23 ) \log t
\end{equation*}
for $t\ge3$,
we deduce from \eqref{eq:change} and Lemma \ref{lem:mainest1_pf1} that
\[
	\norm{\widetilde{u}_{\mathrm{ap}}-u_{\mathrm{ap}}}_{L^2_x}
	\lesssim  (\log t)
\norm{\Jbr{\mu}^{-1} \rho(\mu)(\widetilde\rho(s,\mu)^\frac23 
- \rho(\mu)^\frac23)}_{L^2_\mu(\R^3)}
	\lesssim t^{-\frac56}\log t
\]
and
\[
	\norm{\widetilde{u}_{\mathrm{ap}}-u_{\mathrm{ap}}}_{L^{\frac{10}3}_x}
	\lesssim t^{-\frac35 -\frac56}\log t.
\]
Further, we see from \eqref{eq:trho_est} that
\[
	\norm{\widetilde{v}_{\mathrm{ap}}}_{L^2_x(|x|<t)}
	\lesssim %t^{-1} \norm{(\rho^\frac23 + t^{-\frac13}\Jbr{\mu}^{-1})\rho }_{L^2_\mu} \le 
	t^{-1} \norm{\rho}_{L^{\frac{10}3}_\mu}^{\frac53} + t^{-\frac43}\norm{\rho}_{L^2_\mu}
\]
and
\[
	\norm{\widetilde{v}_{\mathrm{ap}}}_{L^{\frac{10}{3}}_x(|x|<t)}
	\lesssim 
	%t^{-\frac52+\frac{9}{10}} \norm{\Jbr{\mu}(\rho^\frac23 + t^{-\frac13}\Jbr{\mu}^{-1})\rho }_{L^{\frac{10}3}_\mu} \le 
	t^{-\frac35} (t^{-1} \norm{\langle\mu\rangle^{\frac35}
	\rho}_{L^{\frac{50}9}_\mu}^{\frac53} + t^{-\frac43}\norm{\rho}_{L^2_\mu}).
\]
Similarly, we have 
\[
\norm{\nabla_x\widetilde{u}_{\mathrm{ap}}-\nabla_xu_{\mathrm{ap}}}_{L^2_x}
+ \norm{\nabla_x\widetilde{v}_{\mathrm{ap}}}_{L^2_x}
\lesssim
t^{-\frac56}(\log t) \JBR{\norm{(\phi_0,\phi_1)}_{Y}}^\frac73.
\]
Indeed, in view of \eqref{ptx}, the leading term with respect to $t$ appears only when the derivative
$\nabla_x$ hits $e^{in\Jbr{\mu}^{-1}t}$.
Furthermore, in that case, $\nabla_x e^{in\Jbr{\mu}^{-1}t}=-in\mu e^{i\Jbr{\mu}^{-1}t}$
and so the estimate is essentially the same.

Combining these estimates, we conclude that $A$ satisfies \eqref{eq:A3} 
as long as $\gamma<5/6$.
\end{proof}

% \[
% 	v_{n}(s,\mu) := \frac{\l c_n}{n^2-1} t^{-\frac52} \Jbr{\mu}^{\frac52} \widetilde\rho(s,\mu)^{\frac23}\rho(\mu)
% 	\Re (e^{in(\alpha(s,\mu)+\beta(\mu))}).
% \]

\subsection{Proof of \eqref{eq:A2}}\label{subsec:A2}

To complete the proof of Proposition \ref{prop:mainest0},
we prove $A$ satisfies the condition \eqref{eq:A2} for $\gamma<5/6$.
Note that
\[
	(\Box+1)A- N(A) = ((\Box+1)\widetilde{u}_{\mathrm{ap}}-N_{\mathrm{r}})
	+((\Box+1)\widetilde{v}_{\mathrm{ap}}-N_{\mathrm{nr}})
	+(N(\widetilde{u}_{\mathrm{ap}})-N(A)).
\]
The third term of the right hand side is estimated as
\begin{align*}
\|N(A)-N(\widetilde{u}_{\mathrm{ap}})\|_{L_{x}^{2}}
\lesssim{}&
(\|\widetilde{u}_{\mathrm{ap}}\|_{L_{x}^{\infty}}
+\|\widetilde{v}_{\mathrm{ap}}\|_{L_{x}^{\infty}})^\frac23
\|\widetilde{v}_{\mathrm{ap}}\|_{L_{x}^2}\\
\lesssim{}&  t^{-2} \JBR{\norm{\Jbr{\cdot}^3 \rho e^{i\beta}}_{H^2}}^{3}.
\end{align*}
Hence, we estimate the first and the second terms, in what follows.

\begin{proposition}\label{prop:mainest1}
For  $t\ge3$,
\begin{equation}
	\norm{(\Box_{t,x}+1)\widetilde{u}_{\mathrm{ap}} - N_{\mathrm{r}} }_{L^2_x(|x|<t)}\label{eq:1st}
	\lesssim t^{-1-\frac56} (\log t)\JBR{\norm{(\phi_0,\phi_1)}_Y}^{\frac73} 
\end{equation}
holds.
\end{proposition}

\begin{proof}
To show the inequality (\ref{eq:1st}), we begin with 
the computation of the linear part: 
\[
(\Box+1)\widetilde{u}_{\mathrm{ap}}
=\Re\left[(\Box+1)(s^{-\frac32}\Jbr{\mu}^\frac32\rho(\mu)e^{i\beta}
e^{i\langle\mu\rangle^{-1}s+i\widetilde\Psi(s,\mu)\log s})\right].
\]
We split
\begin{equation}\label{eq:main}
\begin{aligned}
&(\Box+1)(s^{-\frac32} \Jbr{\mu}^\frac32\rho(\mu)e^{i\beta}
e^{i\Jbr{\mu}^{-1}s+i\widetilde\Psi(s,\mu)\log s})\\
=&((\Box+1) s^{-\frac32}e^{i\Jbr{\mu}^{-1}s}) \Jbr{\mu}^\frac32\rho(\mu)e^{i\beta} e^{i\widetilde\Psi(s,\mu)\log s}\\
&+s^{-\frac32}e^{i\Jbr{\mu}^{-1}s} \Box( \Jbr{\mu}^\frac32\rho(\mu)e^{i\beta} e^{i\widetilde\Psi(s,\mu)\log s})\\
&+2\pt_t ( s^{-\frac32}e^{i\Jbr{\mu}^{-1}s})\pt_t( \Jbr{\mu}^\frac32\rho(\mu)e^{i\beta} e^{i\widetilde\Psi(s,\mu)\log s})\\
&-2\nabla_x ( s^{-\frac32}e^{i\Jbr{\mu}^{-1}s}) \cdot \nabla_x( \Jbr{\mu}^\frac32\rho(\mu)e^{i\beta} e^{i\widetilde\Psi(s,\mu)\log s})\\
=:& I_1+ I_2 + I_3 + I_4.
\end{aligned}
\end{equation}
In light of Lemma \ref{lem:mainest0}, we have
\begin{equation}\label{eq:In1}
	I_1 = \frac{15}{4} s^{-\frac{7}{2}}\Jbr{\mu}^\frac32\rho(\mu)e^{i(\alpha+\beta)}.
\end{equation}
By (\ref{ab2}), one sees that
\begin{equation}\label{eq:In2}
\begin{aligned}
I_2 
={}&  s^{-\frac32}e^{i\Jbr{\mu}^{-1}s}  
\Jbr{\mu}^\frac32\rho(\mu)e^{i\beta} \pt_s^2 e^{i\widetilde\Psi(s,\mu)\log s}\\
&-2 s^{-\frac52}e^{i\Jbr{\mu}^{-1}s}  \langle \mu \rangle^{2} \mu \cdot \nabla_\mu \pt_s
(\Jbr{\mu}^\frac32\rho(\mu)e^{i\beta}  e^{i\widetilde\Psi(s,\mu)\log s})\\
&- s^{-\frac72}e^{i\Jbr{\mu}^{-1}s}  \Jbr{\mu}^{2} \Delta_\mu
(\Jbr{\mu}^\frac32\rho(\mu)e^{i\beta}  e^{i\widetilde\Psi(s,\mu)\log s}) \\
&-3 s^{-\frac72}e^{i\Jbr{\mu}^{-1}s}  \langle \mu \rangle^{2} \mu \cdot \nabla_\mu
(\Jbr{\mu}^\frac32\rho(\mu)e^{i\beta}  e^{i\widetilde\Psi(s,\mu)\log s}) \\
&- s^{-\frac72}e^{i\Jbr{\mu}^{-1}s}  \Jbr{\mu}^{2} 
\sum_{1\le j,k \le 3 }\mu_j \mu_k \pt_j \pt_k (\Jbr{\mu}^\frac32\rho(\mu)e^{i\beta}  e^{i\widetilde\Psi(s,\mu)\log s})\\
=:{}&  J_1 + J_2 + J_3 + J_0 + J_4.
\end{aligned}
\end{equation}
Moreover, since
\begin{align*}
	\pt_t (s^{-\frac32}e^{i\Jbr{\mu}^{-1}s}) 
	&=i\Jbr{\mu} s^{-\frac32}e^{i\Jbr{\mu}^{-1}s}
	-\frac32 s^{-\frac52}e^{i\Jbr{\mu}^{-1}s},\\
	\nabla_x e^{i\Jbr{\mu}^{-1}s}&=-i \mu e^{i\Jbr{\mu}^{-1}s},
\end{align*}
we have
\begin{equation}\label{eq:In4}
\begin{aligned}
I_{4}
% &{}=2is^{-\frac32}e^{i\Jbr{\mu}^{-1}s} \\
% &{}\quad\quad \times \mu \cdot \{ s^{-1}\langle \mu \rangle \nabla_\mu
% +s^{-1}\langle \mu \rangle\mu (\mu \cdot \nabla_\mu) \}
% ( \Jbr{\mu}^\frac32\rho(\mu)e^{i\beta} e^{i\widetilde\Psi(s,\mu)\log s} )\\
&{}=2is^{-\frac52} e^{i\Jbr{\mu}^{-1}s}
\Jbr{\mu}^3(\mu \cdot \nabla_\mu \Jbr{\mu}^\frac32\rho(\mu)e^{i\beta} e^{i\widetilde\Psi(s,\mu)\log s})
\end{aligned}
\end{equation}
and
\begin{equation}\label{eq:In3}
\begin{aligned}
I_{3}
% &{}=2is^{-\frac32}\Jbr{\mu} e^{i\Jbr{\mu}^{-1}s}
% \{\pt_{s}-s^{-1} \Jbr{\mu}^{2}\mu \cdot \nabla_\mu \}
% (\Jbr{\mu}^\frac32\rho(\mu)e^{i\beta} e^{i\Psi(s,\mu)\log s} )\\
% &\quad{}-3s^{-\frac52}e^{i\Jbr{\mu}^{-1}s}
% \{\pt_{s}-s^{-1} \Jbr{\mu}^{2}\mu \cdot \nabla_\mu \}
% (\Jbr{\mu}^\frac32\rho(\mu)e^{i\beta} e^{i\Psi(s,\mu)\log s} )\\
&{}=-2s^{-\frac52} 
\widetilde\Psi(s,\mu) \Jbr{\mu}^\frac52\rho(\mu)e^{i(\alpha+\beta)}  \\
&{}\quad -2s^{-\frac32}(\log s ) 
\pt_s \widetilde\Psi(s,\mu) \Jbr{\mu}^\frac52\rho(\mu)e^{i(\alpha+\beta)} \\
&\quad{}-2i s^{-\frac52} e^{i\Jbr{\mu}^{-1}s}
\Jbr{\mu}^3(\mu \cdot \nabla_\mu (\Jbr{\mu}^\frac32\rho(\mu)e^{i\beta} e^{i\widetilde\Psi(s,\mu)\log s}))\\
&\quad{}-3is^{-\frac72} 
\widetilde\Psi(s,\mu)\Jbr{\mu}^\frac32\rho(\mu)e^{i(\alpha+\beta)} \\
&\quad{}-3is^{-\frac52}(\log s) 
\pt_s \widetilde\Psi(s,\mu) \Jbr{\mu}^\frac32\rho(\mu)e^{i(\alpha+\beta)}\\
&\quad{}+3s^{-\frac72}  e^{i\Jbr{\mu}^{-1}s}
\Jbr{\mu}^{2} ( \mu \cdot \nabla_\mu
(\Jbr{\mu}^\frac32\rho(\mu)e^{i\beta} e^{i\widetilde\Psi(s,\mu)\log s} ))\\
&{}=:-2s^{-\frac52} 
\widetilde\Psi(s,\mu) \Jbr{\mu}^\frac52\rho(\mu) e^{i(\alpha+\beta)} + J_{5}-I_4+J_{6}+J_{7}-J_{0}.
\end{aligned}
\end{equation}
From (\ref{eq:main}), (\ref{eq:In1}), (\ref{eq:In2}), (\ref{eq:In4}) 
and (\ref{eq:In3}) we reach to
\begin{equation}\label{eq:mainest1_pf2}
\begin{aligned}
	&(\Box +1)u_{\mathrm{ap}} - N_{\mathrm{r}}
	\\
	&\quad = \l c_1 s^{-\frac52}\Jbr{\mu}^\frac52 (\widetilde{\rho}(s,\mu)^\frac23 - \rho(\mu)^\frac23)\rho(\mu)
	\Re(e^{i(\alpha+\beta)})\\
	&\qquad + \Re I_1+\sum_{k=1}^7 \Re J_{k}.
\end{aligned}
\end{equation}
To estimate the right hand side of (\ref{eq:mainest1_pf2}), we show several elementary lemmas.

We will estimate the right hand side of  \eqref{eq:mainest1_pf2} in $L^2_x(|x|<t)$.
Thanks to \eqref{eq:change} and Lemma \ref{lem:mainest1_pf1}, we have
\begin{equation}\label{eq:InS}
	\norm{\l c_1 s^{-\frac52} \Jbr{\mu}^\frac52 (\widetilde\rho(s,\mu)^\frac23 - \rho(\mu)^\frac23) \rho(\mu) e^{i(\alpha+\beta)}}_{L^2_x(|x|<t)}
	\lesssim t^{-1-\frac56}.
\end{equation}
Furthermore, it follows from Lemma \ref{lem:mainest1_pf2} that
\begin{equation}
	\norm{\Re(J_5+J_7)}_{L^2_x(|x|<t)} \lesssim t^{-1-\frac56}\log t.
	\label{eq:j5j7}
\end{equation}
By \eqref{eq:change} and \eqref{eq:In1}, we obtain
\begin{equation}
	\norm{\Re I_1}_{L^2(|x|<t)} \lesssim t^{-2} \tnorm{\Jbr\cdot^{-1} \rho}_{L^2_\mu(\R^3)}
	\le t^{-2} \norm{\rho}_{L^2_\mu}.
	\label{eq:i1}
\end{equation}
Similarly,
\begin{equation}\label{eq:j6}
\begin{aligned}
	\norm{\Re J_6}_{L^2(|x|<t)}
	&{}\lesssim t^{-2} (\tnorm{\Jbr{\cdot}^{-\frac35} 
	\rho}^{\frac53}_{L^{\frac{10}{3}}_\mu(\R^3)}
	+ \tnorm{\Jbr{\cdot}^{-2} \rho}_{L^2_\mu(\R^3)})\\
	&{}\le t^{-2} \norm{\rho}_{H^1}\Jbr{\norm{\rho}_{H^1}}^{\frac23}.
\end{aligned}
\end{equation}

To estimate $J_1$, we note that
\begin{align*}
	&\pt_s^2  e^{i\widetilde\Psi(s,\mu)\log s}\\
	&= i\pt_s^2 \widetilde\Psi(s,\mu) \log s  e^{i \widetilde\Psi(s,\mu)\log s}
	+ 2is^{-1}\pt_s \widetilde\Psi(s,\mu)  e^{i \widetilde\Psi(s,\mu)\log s}\\
	 &\quad - (\pt_s \widetilde\Psi(s,\mu) \log s)^2 e^{i \widetilde\Psi(s,\mu)\log s}
	-2s^{-1} \widetilde\Psi(s,\mu)  \pt_s \widetilde\Psi(s,\mu) \log s  e^{i \widetilde\Psi(s,\mu)\log s}\\
	 &\quad -is^{-2} \widetilde\Psi(s,\mu) e^{i \widetilde\Psi(s,\mu)\log s}
	-s^{-2} \widetilde\Psi(s,\mu)^2 e^{i \widetilde\Psi(s,\mu)\log s}.
\end{align*}
Then, one sees from Lemma \ref{lem:mainest1_pf2} that
\begin{align*}
	\norm{J_1}_{L^2(|x|<t)} \lesssim{}& t^{-2-\frac56}\log t + t^{-2-\frac56}
	+ t^{-2-\frac76}(\log t)^2 \\
	&+ t^{-3}(\log t) \tnorm{\Jbr{\cdot}^{-4}\rho^{\frac13}}_{L^2_\mu(\R^3)}
	%\\&
	+t^{-2} \tnorm{\Jbr{\cdot}^{-1}\widetilde\rho(\cdot,t)^{\frac23}\rho}_{L^2_\mu(\R^3)}\\
	&+t^{-2} \tnorm{\Jbr{\cdot}^{-1}\widetilde\rho(\cdot,t)^{\frac43}\rho}_{L^2_\mu(\R^3)}.
\end{align*}
Using \eqref{eq:trho_est} and the H\"older and the Sobolev inequalities, we obtain
\begin{equation}\label{eq:j1}
\norm{J_1}_{L^2(|x|<t)}
\lesssim t^{-2} \Jbr{\norm{\rho}_{H^1}}^{\frac73}.
\end{equation}

We next estimate $J_3$ and $J_4$. To this end, we remark that
\begin{align*}
	&\pt_j \pt_k (\Jbr{\mu}^\frac32 \rho(\mu)e^{i\beta}  e^{i \widetilde\Psi(s,\mu)\log s})\\
	&{}=(\pt_j \pt_k (\Jbr{\mu}^\frac32 \rho(\mu)e^{i\beta} ))e^{i \widetilde\Psi(s,\mu)\log s}\\ 
	&\quad + i (\pt_j(\Jbr{\mu}^\frac32 \rho(\mu)e^{i\beta} ))  \pt_k \widetilde\Psi(s,\mu)(\log s) e^{i\widetilde\Psi(s,\mu)\log s} \\
	&\quad + i (\pt_k(\Jbr{\mu}^\frac32 \rho(\mu)e^{i\beta} ))  \pt_j \widetilde\Psi(s,\mu)(\log s) e^{i\widetilde\Psi(s,\mu)\log s} \\
	&\quad +i \Jbr{\mu}^\frac32 \rho(\mu)e^{i\beta}  \pt_j \pt_k \widetilde\Psi(s,\mu)(\log s) e^{i \widetilde\Psi(s,\mu)\log s}\\
	&\quad - \Jbr{\mu}^\frac32 \rho(\mu)e^{i\beta}  \pt_j \widetilde\Psi(s,\mu)\pt_k \widetilde\Psi(s,\mu)(\log s)^2 e^{i\widetilde\Psi(s,\mu)\log s}.
\end{align*}
Hence, it follows from Lemma \ref{lem:mainest1_pf3} that
\begin{align*}
	&\norm{J_3+J_4}_{L^2_x(|x|<t)}\\
	&\lesssim t^{-2} \(\norm{\Jbr{\cdot}^{\frac32}\nabla^2(\Jbr{\cdot}^\frac32 \rho e^{i\beta} )}_{L^2_\mu} +	
	(\log t) \norm{\Jbr{\cdot}^3 \rho^\frac23 |\nabla^2 \rho|}_{L^2_\mu} \)\\
% 	&+t^{-2}\log t \norm{\Jbr{\mu}^{-1/2}|\nabla(\Jbr{\mu}^\frac32 \rho(\mu)e^{i\beta})|(A(\mu)^\frac23+1)}_{L^2_\mu(\R^3)}\\
	&\quad +t^{-1-\frac56}(\log t)\( \norm{\Jbr{\cdot}^{\frac32}|\nabla (\Jbr{\cdot}^\frac32 \rho e^{i\beta} )||\Jbr{\cdot}^\frac12\nabla \rho|}_{L^2_\mu} + \norm{\Jbr{\mu}^{\frac72} |\nabla \rho|^2}_{L^2_\mu} \)\\
	&\quad+t^{-\frac{7}3}(\log t )\(\norm{\Jbr{\cdot}^{-\frac12}|\nabla (\Jbr{\cdot}^\frac32 \rho e^{i\beta} )|}_{L^2_\mu}
	+ \norm{\rho}_{L^2}\)\\	
	&\quad + t^{-2} (\log t)^2 \norm{\Jbr{\cdot}^3 |\nabla \rho|^2}_{L^2_\mu}
	+ t^{-3} (\log t)^2 \norm{\Jbr{\cdot}^{-2} \rho^\frac13}_{L^2_\mu}.
\end{align*}
Thus,
\begin{equation}
	\norm{\Re(J_3+J_4)}_{L^2_x(|x|<t)} \lesssim t^{-1-\frac56}(\log t)\JBR{\norm{\Jbr{\cdot}^3 \rho e^{i\beta}}_{H^2}
	+\norm{\Jbr{\cdot}^3 \rho}_{H^2}}^2.
	\label{eq:j2j3}
\end{equation}

Finally, we estimate $J_2$.
Since
\begin{align*}
&\pt_{\mu_j} \pt_s
(\Jbr{\mu}^\frac32\rho(\mu)e^{i\beta}  e^{i \widetilde\Psi(s,\mu)\log s})\\
% ={}& \pt_{\mu_j} 
% (A(\mu)e^{-i\beta} (i\pt_s \widetilde\Psi \log s + is^{-1} \widetilde\Psi) e^{i \widetilde\Psi(s,\mu)\log s})\\
&= (\pt_{\mu_j}( \Jbr{\mu}^\frac32\rho(\mu)e^{i\beta}  )) (i\pt_s \widetilde\Psi(s,\mu) \log s + is^{-1} \widetilde\Psi(s,\mu)) e^{i \widetilde\Psi(s,\mu)\log s}\\
&\quad + \Jbr{\mu}^\frac32\rho(\mu)e^{i\beta}  (i\pt_{\mu_j}\pt_s \widetilde\Psi(s,\mu) \log s + is^{-1} \pt_{\mu_j}\widetilde\Psi(s,\mu)) e^{i \widetilde\Psi(s,\mu)\log s}\\
&\quad - \Jbr{\mu}^\frac32\rho(\mu)e^{i\beta}  (\pt_s \widetilde\Psi(s,\mu) \log s + s^{-1} \widetilde\Psi(s,\mu))\pt_{\mu_j} \widetilde\Psi(s,\mu)(\log s) e^{i \widetilde\Psi(s,\mu)\log s},
\end{align*}
we deduce from Lemmas \ref{lem:mainest1_pf2} and \ref{lem:mainest1_pf3} that
\begin{align*}
	&\norm{J_2}_{L^2_x(|x|<t)}\\
	&\lesssim
	 t^{-1-\frac56} \norm{\Jbr{\cdot}^{5/2}\rho |\nabla \rho| }_{L^2_\mu}
	+t^{-2}(\log t) \norm{ \Jbr{\cdot}^{2} \rho^\frac43 |\nabla \rho| }_{L^2_\mu}\\
	&\quad 
	+t^{-2} \norm{\Jbr{\cdot}^{1/2}|\nabla(\Jbr{\cdot}^{3/2} \rho e^{i\beta})|\rho^\frac23}_{L^2_\mu}
	+t^{-\frac{13}6} (\log t )^2\norm{ \Jbr{\cdot}^{\frac32} \rho |\nabla \rho| }_{L^2_\mu}
	\\
	&\quad +t^{-\frac73} (\log t) (\norm{\rho}_{L^2_\mu}+ \norm{\Jbr{\cdot}|\nabla \rho|}_{L^2_\mu}+
	\tnorm{\Jbr{\cdot}^{-1/2}\nabla(\Jbr{\cdot}^{3/2} \rho e^{i\beta})}_{L^2_\mu}) \\
	&\quad +t^{-\frac73}\(\norm{\rho}_{L^2_\mu}+ \norm{\Jbr{\cdot}^{-1/2}|\nabla(\Jbr{\cdot}^{3/2} \rho e^{i\beta})|}_{L^2_\mu}\)+ t^{-\frac83}(\log t)^2 \norm{\Jbr{\cdot}^{-1}\rho}_{L^2_\mu},
\end{align*}
from which we obtain
\begin{equation}\label{eq:j4}
	\norm{J_2}_{L^2_x(|x|<t)} \lesssim t^{-1-\frac56} \JBR{\norm{\Jbr{\cdot}^3 \rho e^{i\beta}}_{H^2}
	+\norm{\Jbr{\cdot}^3 \rho}_{H^2}}^2.
\end{equation}
Substituting (\ref{eq:InS}), (\ref{eq:j5j7}), (\ref{eq:i1}), 
(\ref{eq:j6}), (\ref{eq:j1}), (\ref{eq:j2j3}), (\ref{eq:j4}) 
into (\ref{eq:mainest1_pf2}), we obtain (\ref{eq:1st})
because $\tnorm{\Jbr{\cdot}^3 \rho e^{i\beta}}_{H^2}
	+\tnorm{\Jbr{\cdot}^3 \rho}_{H^2} \lesssim \norm{(\phi_0,\phi_1)}_{Y}$.
This completes the proof of Proposition \ref{prop:mainest1}. 
\end{proof}

Next we give an estimate for difference between 
$(\Box+1)\widetilde{v}_{\mathrm{ap}}$ and the non-resonance part 
$N_{\mathrm{nr}}$. 

\begin{proposition}\label{prop:mainest2}
For all $n\ge2$ and $t\ge 3$, we have 
\begin{equation}
	\norm{(\Box+1)v_n - N_n}_{L^2_x(|x|< t)} \lesssim |c_n| t^{-2} \JBR{\norm{(\phi_0,\phi_1)}_Y}^{\frac73}.
	\label{eq:2nd}
\end{equation}
In particular,
\[
	\norm{(\Box+1)\widetilde{v}_{\mathrm{ap}} - N_{\mathrm{nr}}}_{L^2_x(|x|< t)} \lesssim t^{-2} \JBR{\norm{(\phi_0,\phi_1)}_Y}^{\frac73} .
\]
\end{proposition}

\begin{proof}
Denoting $d_n=-\l c_n/(n^2-1)$, we have
\[
(\Box+1)v_n
=d_n \Re (\Box+1)(s^{-\frac52}\Jbr{\mu}^{\frac52}\widetilde\rho(s,\mu)^{\frac23}\rho(\mu)e^{in\beta}
e^{in\langle\mu\rangle^{-1}s+in \widetilde\Psi(\mu,t)\log s}).
\]
As in the previous case, we split 
\begin{equation}\label{eq:In100}
\begin{aligned}
&d_n(\Box+1)(s^{-\frac52}\Jbr{\mu}^{\frac52} \widetilde\rho(s,\mu)^{\frac23}\rho(\mu)e^{in\beta}
e^{in\Jbr{\mu}^{-1}s+in\widetilde\Psi(s,\mu)\log s})\\
&=d_n((\Box+1) s^{-\frac52}e^{in\Jbr{\mu}^{-1}s})
\Jbr{\mu}^\frac52\widetilde\rho (s,\mu)^{\frac23}\rho(\mu)e^{in\beta} e^{in\widetilde\Psi(s,\mu)\log s}\\
&\quad + d_n s^{-\frac52}e^{in\Jbr{\mu}^{-1}s} 
\Box(\Jbr{\mu}^\frac52\widetilde\rho (s,\mu)^{\frac23}\rho(\mu)e^{in\beta} e^{in\widetilde\Psi(s,\mu)\log s})\\
&\quad + 2d_n\pt_t ( s^{-\frac52}e^{in\Jbr{\mu}^{-1}s})
\pt_t(\Jbr{\mu}^\frac52\widetilde\rho (s,\mu)^{\frac23}\rho(\mu)e^{in\beta} e^{in\widetilde\Psi(s,\mu)\log s})\\
&\quad-2d_n\nabla_x ( s^{-\frac52}e^{in\Jbr{\mu}^{-1}s}) \cdot 
\nabla_x(\Jbr{\mu}^\frac52\widetilde\rho (s,\mu)^{\frac23}\rho(\mu)e^{in\beta} e^{in\widetilde\Psi(s,\mu)\log s})\\
&=: I_{1,n}+ I_{2,n} + I_{3,n} + I_{4,n}.
\end{aligned}
\end{equation}
By means of Lemma \ref{lem:mainest0},
\begin{equation}\label{eq:mainest1_pf21}
\begin{aligned}
	I_{1,n} &= 
	\l c_n s^{-\frac52} 
	\Jbr{\mu}^\frac52\widetilde\rho (s,\mu)^{\frac23}\rho(\mu) e^{in(\alpha+\beta)}\\
	&\quad -2in d_n s^{-\frac72} 
	\Jbr{\mu}^\frac72 \widetilde\rho (s,\mu)^{\frac23}\rho(\mu) e^{in(\alpha+\beta)}\\
	&\quad +\frac{35d_n}{4} s^{-\frac{9}{2}}
	\Jbr{\mu}^\frac52\widetilde\rho (s,\mu)^{\frac23}\rho(\mu) e^{in(\alpha+\beta)}\\
	&=:\l c_n s^{-\frac52}
	\Jbr{\mu}^\frac52\widetilde\rho (s,\mu)^{\frac23}\rho(\mu) e^{in(\alpha+\beta)}
	+K_{1,n} + K_{2,n}.
\end{aligned}
\end{equation}
In a similar way, we have
\begin{equation}\label{eq:mainest1_pf22}
\begin{aligned}
I_{2,n} 
={}& d_n s^{-\frac52}e^{in\Jbr{\mu}^{-1}s}  
\Jbr{\mu}^{\frac52}\rho(\mu)e^{in\beta} \pt_s^2(\widetilde\rho(s,\mu)^{\frac23}e^{in\widetilde\Psi(s,\mu)\log s})\\
&-2 d_n s^{-\frac72}e^{in\Jbr{\mu}^{-1}s}  \langle \mu \rangle^{2}
\mu \cdot \nabla_\mu \pt_s
(\Jbr{\mu}^{\frac52} \widetilde\rho(s,\mu)^{\frac23}\rho(\mu)e^{in\beta}  e^{in\widetilde\Psi(s,\mu)\log s})\\
&- d_ns^{-\frac92}e^{in\Jbr{\mu}^{-1}s}  \Jbr{\mu}^{2} \Delta_\mu
(\Jbr{\mu}^{\frac52} \widetilde\rho(s,\mu)^{\frac23}\rho(\mu)e^{in\beta}  e^{in\widetilde\Psi(s,\mu)\log s}) \\
&-3 d_ns^{-\frac92}e^{in\Jbr{\mu}^{-1}s}  \langle \mu \rangle^{2} \mu \cdot \nabla_\mu
(\Jbr{\mu}^{\frac52} \widetilde\rho(s,\mu)^{\frac23}\rho(\mu)e^{in\beta} e^{in\widetilde\Psi(s,\mu)\log s}) \\
&- d_n s^{-\frac92}e^{in\Jbr{\mu}^{-1}s}  \Jbr{\mu}^{2}
\sum_{1\le i,j \le 3 }\mu_i \mu_j \pt_i \pt_j (
\Jbr{\mu}^{\frac52} \widetilde\rho(s,\mu)^{\frac23}\rho(\mu)e^{in\beta}  e^{in\widetilde\Psi(s,\mu)\log s})\\
=:{}&  J_{1,n} + J_{2,n} + J_{3,n} - \frac32 K_{3,n}+ J_{4,n}.
\end{aligned}
\end{equation}
By using the identities
\begin{align*}
	\pt_t (s^{-\frac52}e^{in\Jbr{\mu}^{-1}s}) 
	&=in\Jbr{\mu} s^{-\frac52}e^{in\Jbr{\mu}^{-1}s}
	-\frac52 s^{-\frac72}e^{in\Jbr{\mu}^{-1}s},\\
	\nabla_x e^{in\Jbr{\mu}^{-1}s}&= -i n\mu e^{in\Jbr{\mu}^{-1}s},
\end{align*}
we obtain
\begin{equation}\label{eq:mainest1_pf24}
\begin{aligned}
I_{4,n}
&{}=2ind_ns^{-\frac52}e^{in\Jbr{\mu}^{-1}s} \\
&{}\quad\quad \times \mu \cdot \{ s^{-1}\langle \mu \rangle \nabla_\mu
+s^{-1}\langle \mu \rangle\mu (\mu \cdot \nabla_\mu)\}
\\&{}\quad\quad\quad\quad
(\Jbr{\mu}^{\frac52} \widetilde\rho(s,\mu)^{\frac23}\rho(\mu)e^{in\beta}e^{in\widetilde\Psi(s,\mu)\log s} )\\
&{}=2ind_ns^{-\frac72} e^{in\Jbr{\mu}^{-1}s}
\Jbr{\mu}^3\\
&{}\quad\quad \times(\mu \cdot \nabla_\mu (\Jbr{\mu}^{\frac52} \widetilde\rho(s,\mu)^{\frac23}\rho(\mu)e^{in\beta} e^{in\widetilde\Psi(s,\mu)\log s}))
\end{aligned}
\end{equation}
and
\begin{equation}\label{eq:mainest1_pf23}
\begin{aligned}
I_{3,n}
% &{}=2ind_ns^{-\frac52}\Jbr{\mu} e^{in\Jbr{\mu}^{-1}s}\\
% &\quad\qquad\times
% \{\pt_{s}-s^{-1} \Jbr{\mu}^{2}\mu \cdot \nabla_\mu \}
% (\Jbr{\mu}^{\frac52} \widetilde\rho(s,\mu)^{\frac23}\rho(\mu)e^{in\beta} e^{in\widetilde\Psi(s,\mu)\log s} )\\
% &\quad{}-5d_ns^{-\frac72}e^{in\Jbr{\mu}^{-1}s}\\
% &\quad\qquad\times
% \{\pt_{s}-s^{-1} \Jbr{\mu}^{2}\mu \cdot \nabla_\mu \}
% (\Jbr{\mu}^{\frac52} \widetilde\rho(s,\mu)^{\frac23}\rho(\mu)e^{in\beta} e^{in\widetilde\Psi(s,\mu)\log s} )\\
&{}=-2n^2d_ns^{-\frac72}\Jbr{\mu} e^{in\Jbr{\mu}^{-1}s}
\widetilde\Psi(s,\mu) \Jbr{\mu}^{\frac52} \widetilde\rho(s,\mu)^{\frac23}\rho(\mu)e^{in\beta} e^{in\widetilde\Psi(s,\mu)\log s} \\
&{}\quad -2n^2d_ns^{-\frac52}(\log s )\Jbr{\mu} e^{in\Jbr{\mu}^{-1}s}\\
&\quad\qquad\times
\pt_s \widetilde\Psi(s,\mu) \Jbr{\mu}^{\frac52} \widetilde\rho(s,\mu)^{\frac23}\rho(\mu)e^{in\beta} e^{i\Psi(s,\mu)\log s} \\
&{}\quad +
2ind_ns^{-\frac52}\Jbr{\mu} e^{in\Jbr{\mu}^{-1}s}
\Jbr{\mu}^{\frac52} (\pt_s \widetilde\rho(s,\mu)^{\frac23})\rho(\mu)e^{in\beta} e^{in\widetilde\Psi(s,\mu)\log s} \\
&\quad{}-2i nd_ns^{-\frac72} e^{in\Jbr{\mu}^{-1}s}\\
&\quad\qquad\times
\Jbr{\mu}^3(\mu \cdot \nabla_\mu 
(\Jbr{\mu}^{\frac52} \widetilde\rho(s,\mu)^{\frac23}\rho(\mu)e^{in\beta} e^{in\widetilde\Psi(s,\mu)\log s}))\\
&\quad{}-5ind_ns^{-\frac92}  e^{in\Jbr{\mu}^{-1}s} \widetilde\Psi(s,\mu)
\Jbr{\mu}^{\frac52} \widetilde\rho(s,\mu)^{\frac23}\rho(\mu)e^{in\beta} e^{in\widetilde\Psi(s,\mu)\log s} \\
&\quad{}-5ind_ns^{-\frac72}(\log s) e^{in\Jbr{\mu}^{-1}s}\\
&\quad\qquad\times
\pt_s \widetilde\Psi(s,\mu) \Jbr{\mu}^{\frac52} \widetilde\rho(s,\mu)^{\frac23}\rho(\mu)e^{in\beta} e^{in\widetilde\Psi(s,\mu)\log s}\\
&\quad{}-5d_ns^{-\frac72}e^{in\Jbr{\mu}^{-1}s}
\Jbr{\mu}^{\frac52} (\pt_s \widetilde\rho(s,\mu)^{\frac23})\rho(\mu)e^{in\beta} e^{in\widetilde\Psi(s,\mu)\log s}\\
&\quad{}
+5d_ns^{-\frac92}  e^{in\Jbr{\mu}^{-1}s}\Jbr{\mu}^{2} ( \mu \cdot \nabla_\mu
(\Jbr{\mu}^{\frac52} \widetilde\rho(s,\mu)^{\frac23}\rho(\mu)e^{in\beta} e^{in\widetilde\Psi(s,\mu)\log s} ))\\
&{}=:K_{4,n} + J_{5,n}+K_{5,n}-I_{4,n}+J_{6,n}+J_{7,n}+K_{6,n}+\frac52 K_{3,n}.
\end{aligned}
\end{equation}
Substituting
(\ref{eq:mainest1_pf21}), (\ref{eq:mainest1_pf22}), 
(\ref{eq:mainest1_pf24}) and (\ref{eq:mainest1_pf23}) into (\ref{eq:In100}), 
we conclude that
\begin{align}
	(\Box +1)v_n - N_n\label{eq:mainest1_pf5}
	={}&\lambda c_{n}s^{-\frac52}\Jbr{\mu}^\frac52
	\left(\widetilde\rho(s,\mu)^{\frac23}-\rho(\mu)^{\frac23}\right)\rho(\mu)\Re
	(e^{in(\alpha+\beta)})\\
	& +\sum_{k=1}^7 \Re J_{k,n} + \sum_{k=1}^6 \Re K_{k,n}.\nonumber
\end{align}
The first term of the right hand side is $O(|c_n|t^{-1-5/6})$ in $L^2_x(|x|<t)$ with the help of \eqref{eq:change} and Lemma \ref{lem:mainest1_pf1}.
The estimates for $J_{k,n}$ are similar to those for the corresponding $J_{k}$.
The difference are that the additional decay effect of order $O(t^{-1})$ make them all higher order terms,
that each term is multiplied by $\Jbr{\mu} \widetilde\rho(s,\mu)^{2/3}$,
and that the order in $n$ is at most $O(n^2|d_n|)=O(|c_n|)=O(n^{-8/3})$ as $n\to\I$ because
the phase parts are differentiated at most twice.
The terms $K_{k,n}$ are new but
the estimates for $K_{k,n}$ are done in a similar way.
% The leading term is $K_{2,n}$. 
This completes the proof of Proposition \ref{prop:mainest2}.
\end{proof}

\vskip3mm
\noindent {\bf Acknowledgments.} 
S.M. is partially supported by the Sumitomo Foundation, Basic Science Research
Projects No.\ 161145. 
%The authors would like to thank Professor Hideaki Sunagawa for 
%drawing their attention to his related works. 
J.S. is partially supported by JSPS, %MEXT,
Grant-in-Aid for Young Scientists (A) 25707004.

\end{document}